\documentclass[a4paper,10pt]{article}
\usepackage[centertags]{amsmath}
\usepackage{amsfonts}
\usepackage{amssymb}
\usepackage{amsthm}
\usepackage{dsfont}
\usepackage{mathrsfs}
\usepackage{wasysym}
\usepackage{esint}
\usepackage[usenames]{color}
\usepackage{url}
\usepackage
{hyperref}
\theoremstyle{plain}
\newtheorem{lemma}{Lemma}[section]
\newtheorem{theorem}[lemma]{Theorem}

\newtheorem{cor}[lemma]{Corollary}

\theoremstyle{definition}

\newtheorem{problem}[lemma]{Problem}

\newtheorem{remark}[lemma]{Remark}

\newcommand{\twomat}[4]{ \left( \begin{array}{cc}
#1 & #2 \\
#3 & #4
\end{array} \right)}

\newcommand{\pl}[2]{{\frac{\partial #1}{\partial #2}}}

\newcommand{\zb}{\overline{z}}
\newcommand{\db}{\overline{\partial}}
\newcommand{\de}{\partial}
\newcommand{\R}{\mathbb{R}}
\newcommand{\C}{\mathbb{C}}

\newcommand{\F}{\mathbb{F}}
\newcommand{\h}{\mathcal{H}}
\newcommand{\A}{\mathcal{A}}

\newcommand{\tu}{\rightarrow}
\newcommand{\uu}{\uparrow}

\newcommand{\wtu}{\rightharpoonup}

\newcommand{\emb}{\hookrightarrow}
\newcommand{\cemb}{\subset \subset}
\newcommand{\In}{\subset}

\newcommand{\om}{\omega}
\newcommand{\dl}{{\delta}}
\newcommand{\Dl}{{\Delta}}

\newcommand{\al}{{\alpha}}

\newcommand{\ed}{{\rm d}}
\newcommand{\id}{\,\,\ed}
\newcommand{\dv}{{\rm d}^{\ast}}
\newcommand{\D}{{\nabla}}
\newcommand{\Db}{{\nabla^{\bot}}}
\newcommand{\ti}[1]{{\tilde{#1}}}
\newcommand{\eps}{{\varepsilon}}
\newcommand{\fr}[2]{\frac{#1}{#2}}

\newcommand{\Nn}{\mathcal{N}}
\newcommand{\T}{{\rm T}}
\newcommand{\N}{{\rm N}}

\newcommand{\vlinesub}[1]{\vline_{_{_{_{_{_{_{_{_{{#1}=0}}}}}}}}}}
\begin{document}
\title{Critical $\overline{\partial}$ problems in one complex dimension and some remarks on conformally invariant variational problems in two real dimensions.}
\date{\today}
\author{Ben Sharp}
\pagestyle{myheadings}
\markboth{{\sc Critical $\db$ problems}}{{\sc Critical $\db$ problems}}
\maketitle
\begin{abstract}
We will study a linear first order system, a connection $\overline{\partial}$ problem, on a vector bundle equipped with a connection, over a Riemann surface. We show optimal conditions on the connection forms which allow one to find a holomorphic frame, or in other words to prove the optimal regularity of our solution. The underlying geometric principle, discovered by Koszul-Malgrange, is classical and well known; it gives necessary and sufficient conditions for a connection to induce a holomorphic structure on a vector bundle over a complex manifold. Here we explore the limits of this statement when the connection is not smooth and our findings lead to a very short proof of the regularity of harmonic maps in two dimensions as well as re-proving a recent estimate of Lamm and Lin concerning conformally invariant variational problems in two dimensions. 
\end{abstract}

\section{Introduction} 

We will consider a square integrable connection on a smooth vector bundle $E^m$ over a Riemann surface $\Sigma$. 
Our vector bundle may be real or complex, however the problems we wish to consider will largely require us to complexify $E$ when it is real (unless the connection happens to be flat). Since we are working over a Riemann surface we may consider the related $\db$-problem associated to sections of $\wedge^{(1,0)}\T^{\ast}\Sigma\otimes E$. We ask: Under what circumstances can we locally find a holomorphic frame? Or, can we find a cover of $\Sigma$ with a collection of bundle trivialisations such that the transition charts are holomorphic? The latter question is equivalent to being able to find a holomorphic frame over each trivialisation. 

Since the question is of a local nature, we work with a small piece of $\Sigma$ over which $E$ is trivial, therefore we may simply consider the case $\Sigma =D\In \C$ the unit disc and $E=D\times \F^m$ where $\F= \R$ or $\C$. Now the connection is defined entirely by one-forms which we denote by $\om\in L^2(D , gl(m,\C)\otimes\wedge^1\T^{\ast}\R^2)$. The frame $S:D\tu GL(m,\C)$ that we are going to look for will solve \footnote{Initially we consider $\om$ as a $gl(m,\C)$-valued real one form, i.e. $\om = \om^x\ed x + \om^y \ed y$ with $\om^x ,\om^y :D\to gl(m,\C)$; but we can just as easily express $\om$ with respect to $\ed z$ and $\ed\zb$, at which point $\om^{\zb}:= \fr12 (\om^x+i\om^y)\ed\zb=\om^{(0,1)}$ is the $(0,1)$ part of $\om$.}
\begin{equation}\label{eqn S}
\db S =- \om^{\zb} S
\end{equation}
or, to put it another way 
\begin{equation*}
(S^{-1}\ed S + S^{-1}\om S)^{(0,1)} = 0.\footnote{The notation $(p,q)$ will  refer either to $p+q$ forms of type $\ed z^1 \wedge \dots \wedge \ed z^p \wedge \ed \zb^1 \wedge \dots \wedge \ed \zb^q$ or if $G$ is a $p+q$ form then $G^{(p,q)}$ is the projection of $G$ onto forms of type $(p,q)$. }
\end{equation*}
In words, we could say that re-writing the connection forms with respect to this new frame (or trivialisation) $S$ forces the $(0,1)$-part of the new connection forms to be zero. A classical theorem of Koszul-Malgrange \cite[Theorem 1]{KM} c.f. Theorem \ref{theorem KM} tells us that when $\om$ is smooth then we can find $S$ if and only if $F_{\om}^{(0,2)}\equiv 0$ where $$F=\ed \om + [\om,\om]\footnote{$[\om,\om]$ is a two form given by $[\om,\om](X,Y) = [\om(X),\om(Y)]$ where $[,]$ is the Lie bracket of $gl(m,\C)$.}$$ is the curvature of our connection \footnote{Perhaps the reader should compare with the analogous statement in the real setting; that one can find a parallel frame solving 
$$S^{-1}\ed S + S^{-1}\om S =0$$
if and only if $F\equiv 0$, i.e. when the connection is flat.}. Thus in our setting, under the assumption that $\om$ is smooth, one can \emph{always} find such a frame $S$. However this stops being true when we only assume $\om\in L^2$ (or even $\om\in L^{2,q}$ for all $q>1$, see section \ref{complex counter}). 
The problem is that we do not get \emph{a-priori} $L^{\infty}$ estimates for $S$ when $\om\in L^2$, which makes it impossible to guarantee that we can find an invertible matrix $S$ solving \eqref{eqn S}. However if we assume that $\|\om\|_{L^{2,1}}$ is sufficiently small then we can find $S$ via a fixed point argument and ensure that it stays a bounded distance from the identity, Theorem \ref{theorem KM L21}. It follows that $\om \in L^p$ for $p>2$ also works and we can consider $\om \in L^2$ as being a borderline case that fails to hold. Therefore we will assume a further structural condition on $\om\in L^2$ that will ensure the existence of $S$, Theorem \ref{dbar}.

A corollary of Theorem \ref{dbar} allows one to prove regularity {for} maps (or sections) $\al \in L^2(D,\C^m\otimes \wedge^{(1,0)}\T^{\ast}\C)$ solving 
\begin{equation}\label{eqn: dbar}
\db_{\omega}(\al)= \db \al + \om^{\zb}\wedge \al = 0,
\end{equation}
in particular one can show that with $S$ solving \eqref{eqn S}, we have 
$$\db(S^{-1}\al) =0 $$
and the highest regularity of $\al$ we can expect is the same as that of $S$. We remark that the PDE \eqref{eqn: dbar} is critical in the sense that we have $\db \al \in L^1$ so with standard Calderon-Zygmund estimates we can conclude that $\D \al \in L^{1,\infty}$ (a space with $L^1$ as a strict subset) i.e. that
$$|\{z\in D : |\D \al(z)|>s\}|\leq Cs^{-1}.$$ 
However we will show that we can find such an $S\in L^{\infty}\cap W^{1,2}$ at which point these estimates pass locally onto $\al$. In fact we also end up with an estimate for $|\al|^2$ in the local hardy space $h^1$ on the whole disc. The regularity for $\al$ is therefore much higher than we would expect, due to the geometric nature of the problem. Essentially $\al$ is geometrically holomorphic and when the geometry is `sufficiently nice' we can understand it to be locally genuinely holomorphic.

This theory is closely related to H\'elein's \cite{helein_conservation} regularity theory for harmonic maps from a Riemann surface to a closed Riemannian manifold $\Nn$; indeed it provides a short proof for the full regularity theory in two dimensions using \emph{only Wente-type estimates and Coulomb gauge methods} without requiring that $\T\Nn$ be trivial. The assumption that $\T\Nn$ be trivial can be made without loss of generality if $\Nn$ is sufficiently regular: When $\Nn$ is $C^4$ H\'elein proved that there is a totally geodesic embedding of $\Nn$ into a torus, thus harmonic maps into $\Nn$ lift to harmonic maps into a torus (and we may therefore consider only targets with  trivial tangent bundle). Theorem \ref{dbar} allows us to side-step this technicality {and we require the minimal regularity assumptions on $\Nn$, that it is a  $C^2$ submanifold of $\R^m$ with bounded second fundamental form (which follows trivially by the first assumption if $\Nn$ is closed)}.  We mention here that H\'elein's theory simplifies if $\Nn$ is $C^2$ with trivial tangent bundle. Under the assumption that the tangent bundle is trivial, one can employ Coulomb gauge methods and write the harmonic map equation as \eqref{eqn: dbar} with $\om^{\zb}\in L^{2,1}$ as is done in \cite{helein_conservation}. The Lorentz space $L^{2,1}$ is strictly contained in $L^2$ but contains $L^p$ for $p>2$ on bounded domains. A function $f\in L^{2,1}$ if and only if 
$$\int |\{z : |f(z)|>t\}|^{\fr 12}\id t <\infty.$$

The theory here is also related to the work of Rivi\`ere \cite{riviere_inventiones}, who generalised the regularity theory of H\'elein, where he proves the full regularity for critical points of conformally invariant elliptic Lagrangians in two dimensions by considering a geometric divergence problem (vs a geometric $\db$ problem). Specifically he considered maps $u\in W^{1,2}(B_1,\R^m)$ weakly solving 
\begin{equation}\label{eqn: riviere}
0=\dv_{\om}(\ed u) = \dv \ed u -\ast(\om\wedge\ast\ed u) = -\Dl u - \om.\D u
\end{equation}
for $\om \in L^2(B_1, so(m)\otimes \wedge^1 \T^{\ast}\R^2)$ on the unit ball $B_1\In \R^2$. This PDE is critical in the sense that the best one can do with straight forward elliptic estimates is to get estimates on $\D u$ in $L^{2,\infty}$ (i.e. $|\D u|^2 \in L^{1,\infty}$). However Rivi\`ere proved the existence of a frame $A\in L^{\infty}\cap W^{1,2}(B_1,GL(m,\R))$ (a perturbed Coulomb gauge) such that 
\begin{equation}\label{eqn: riviere gauge}
\dv (\ed A - A\om) =0
\end{equation}
when $\|\om\|_{L^2}$ is sufficiently small. This enables one to re-write \eqref{eqn: riviere} and uncovers hidden Jacobian determinant terms. By using classical Wente estimates one can show that $\D u \in L^p$ for every $p<\infty$ (see \cite{Sh_To} or \cite{riviere_laurain}). Finally he observed that critical points of conformally invariant elliptic Lagrangians in two dimensions solve a PDE of the form \eqref{eqn: riviere} to conclude the full regularity of solutions under the weakest regularity assumptions on the Lagrangian. In particular one can conclude regularity of harmonic maps into $C^2$ targets, or the regularity of conformal immersions of the disc in $\R^3$ with bounded mean curvature and finite area.

Going back to \eqref{eqn: riviere gauge} we could conclude that there exists a matrix $B\in W^{1,2}(B_1, gl(m,\R))$ such that
$$\ed A - \ast \ed B = A\om$$
or, to write it another way 
$$\db (A-iB) = A\om^{\zb}$$
(compare this with \eqref{eqn S} and remember that we cannot control $\|B\|_{L^{\infty}}$ and therefore there is no reason that $A-iB$ be invertible). 

In section \ref{sec: applications} we show that critical points of conformally invariant elliptic Lagrangians solve \eqref{eqn: dbar} and, under an added regularity assumption, we can find the frame $S$ solving \eqref{eqn S}. Unlike the case for harmonic maps we require the theory of Rivi\`ere, namely the existence of the frame $A$ solving \eqref{eqn: riviere gauge}, in order to find $S$. However this can still be used to re-prove a recent estimate of {Lamm} and Lin \cite{lamm_lin}, Theorem \ref{thm: lamm_lin}. We remark that it might be possible to drop the {added regularity} and still be able to find $S$ in this setting; either a positive or a negative answer to this question would provide further insight into these regularity problems.  

Another interesting problem would be to extend this theory to higher dimensional complex domains, but of course one would have to impose the condition that $F_{\omega}^{(0,2)} = 0$ in a weak sense (which is given for free in one dimension), and find the right borderline spaces for $\om$ to lie in. The author does not know of any geometric situation where the higher dimensional theory would apply.

\emph{Acknowledgements:} The author was supported by the Leverhulme trust, and the work was completed with funding from the European Research Council. The author would also like to thank Mario Micallef for useful discussions.

\section{Results}

Since the PDE we are trying to solve only involves $\om^{\zb}$ we may  consider local connection forms $\om \in L^2(D,u(m)\otimes\wedge^1 \T^{\ast}\R^2)$ without loss of generality (see Remark \ref{rem general to unitary}).  
The assumptions we want to impose are that such $\om$ admit the following Hodge decomposition: 
$$\om = \ed a + \dv b$$
with $a\in W^{1,2}$, $b\in W_0^{1,(2,1)}$. 
We say that such an $\om$ satisfies condition $\dagger$. If for some $\eps$ we know that 
$$\|\om\|_{L^2(D)} + \|\D b\|_{L^{2,1}(D)} \leq \eps$$
then we say that $\om$ satisfies condition $\dagger_{\eps}$. Another way of writing this condition is that $\om \in L^2(D,u(m)\otimes\wedge^1 \T^{\ast}\R^2)$ satisfies 
$$\D \Dl^{-1} (\ed \omega) \in L^{2,1}$$
with 
$$\|\om\|_{L^2(D)} + \|\D \Dl^{-1}(\ed \omega)\|_{L^{2,1}(D)} \leq \eps.$$

\begin{theorem}\label{dbar}
There exists $\eps >0$ such that whenever $\om\in L^2(D,u(m)\otimes\wedge^1\T^{\ast}\R^2)$ satisfies condition $\dagger_{\eps}$, 
there exists a change of frame $S\in L^{\infty}\cap W_{loc}^{1,2}(D,Gl(m,\C))$ such that $$\db S = -\om^{\zb} S$$
with
$$\|\rm{dist}(S,U(m))\|_{L^{\infty}(D)} \leq \fr13$$
and for any $U\cemb D$ there exists some $C=C(U)<\infty$ such that
$$ \|\D S\|_{L^2(U)} \leq C\|\om\|_{L^2}.$$
\end{theorem}

Such an $S$ is called a holomorphic frame. Going back to the general case of a smooth vector bundle $(E,\pi)$ over a Riemann surface $\Sigma$ with an $L^2$ connection $D_E$, if we could find a cover of $\Sigma$, $\{U_i\}$ such that the connection forms over each $U_i$ satisfy condition $\dagger_{\eps}$, then we can skew our trivialisations by the non-smooth changes of frame $S_i$. In other words where we had a smooth diffeomorphisms $\phi_i = (\pi,\varphi_i) : \pi^{-1}(U_i)\to U_i \times \C^m$, we replace them by non-smooth $\tilde{\phi}_i(e)=(\pi(e), S_i(\pi(e)){{\varphi}_i(e))}$. The reader can check that the new transition charts $\tilde{\phi}_{ij}=S_i\phi_{ij}S_j^{-1}$ will be holomorphic\footnote{Compare with the flat scenario, where by skewing our trivialisations by parallel frames yields locally constant (rather than holomorphic) transition functions.}
 (at the expense of the trivialisations being non-smooth).
\begin{cor}\label{cor_dbar}
Let $\al\in L^2(D,\C^m\otimes \wedge^{(1,0)}\T^{\ast}_{\C}\R^2)$ and $\om \in L^2(D,u(m) \otimes \wedge^1 \T^{\ast} \R^2)$. Suppose that $\om$ satisfies condition $\dagger$ and that $\db_{\omega} \al = 0$, i.e. 
$$\db \al = -\om^{\zb} \wedge \al $$
then $\al \in L^{\infty}\cap W^{1,2} (U)$ for all $U\cemb D$. There exists an $\eps >0$ such that if $\om$ satisfies $\dagger_{\eps}$ then there exists a $C=C(U)<\infty$ such that 
$$\|\al\|_{L^{\infty}(U)} \leq C \|\al\|_{L^1}$$
and
$$ \|\D \al\|_{L^2(U)} \leq C\|\al\|_{L^1}(1+\|\om\|_{L^2}).$$
Moreover, under these assumptions we have $|\al|^2 \in h^1(D)$ with 
$$\||\al|^2\|_{h^1(D)} \leq C\|\al\|_{L^2(D)}^2.$$
\end{cor}
As per the introduction, re-writing everything in terms of $S$ means the connection $\db$ problem is a `genuine' $\db$ problem. The space $h^1$ is the local Hardy space, see for instance \cite[Appendix A.2]{Sh_To} for a brief introduction or \cite{goldberg}. 
\begin{remark}\label{rem general to unitary}
We remark here that given any $\hat{\omega}\in L^2(D,gl(m,\C)\otimes \wedge^1\T^{\ast}\R^2)$ we can always find a unique $\om \in L^2(D, u(m) \otimes\wedge^1 \T^{\ast}\R^2)$ such that $\om^{\zb} = \hat{\omega}^{\zb}$. Indeed if we write 
$$\hat{\omega}^{\zb} =( \hat{\omega}_1 + i\hat{\omega}_2) \ed \zb $$
with $\hat{\omega}_j: D \tu gl(m,\R)$. Then we can decompose each $\hat{\omega}_j$ into its symmetric and antisymmetric part,
$$\hat{\omega}_j = \hat{\omega}_j^S + \hat{\omega}_j^A$$
thus letting $\om^x = 2(\hat{\omega}_1^A + i\hat{\omega}_2^S) : D \tu u(m)$ and $\om^y = 2(\hat{\omega}_2^A  -i\hat{\omega}_1^S) : D \tu u(m)$ and
$$\om = \om^x \ed x + \om^y \ed y \in  L^2(D, u(m) \otimes \wedge^1\T^{\ast}\R^2)$$
we have 
$$\om^{\zb} = \fr12 (\om^x + i\om^y)\ed \zb = \hat{\omega}^{\zb}.$$

Therefore for any such $\hat{\omega}$ we can apply Theorem \ref{dbar} and Corollary \ref{cor_dbar} if $\om$ satisfies condition $\dagger_{\eps}$.
\end{remark}

\section{Applications to Harmonic maps and conformally invariant Lagrangians}\label{sec: applications}
\subsection{Harmonic maps}
When one considers a harmonic function $u:U\In\R^n\to\R^m$ there are a few equivalent viewpoints that can be used to understand the PDE that is solved. Harmonic functions $u$, are critical points of the Dirichlet energy 
$$E(v):=\fr12\int_U |\D v|^2 \id x,$$
which is equivalent to $u$ being a solution to 
$$-\Dl u = -\rm{div}(\D u) = \dv \ed u = 0.$$
One might also consider the PDE not in terms of $u$, but $\ed u$, and rather pedantically write the coupled system
$$\dv (\ed u) = 0 \,\,\,\,\,\,\,\,\,\,\,\,\,\text{and}\,\,\,\,\,\,\,\,\,\,\,\,\,\,\,\ed (\ed u) = 0$$
at which point we could say that $\ed u$ is a harmonic one form, or equivalently that each of the $m$ functions $H^i = (u^i_{x^1},\dots,u^i_{x^n}):U\to\R^n$ solve the Cauchy-Riemann equations. When $n=2$ we can more succinctly write this as 
$$\db (\de u) = 0$$
where we are now considering $U\In \C$ and $\ed = \de + \db$ is the usual splitting i.e. $\de v = \pl{v}{z}\ed z$ and similarly for $\db$. 

From now on we will restrict to considering two dimensional domains and the target a Riemannian manifold $(\Nn, h)$. Due to the conformal invariance of the problems we are looking at, we take the unit disc $B_1\In \R^2$ with the Euclidean metric as our domain. In order to be able to write down the PDE appearing below we will be implicitly using coordinates on $\Nn$ and therefore we are assuming that $u$ is at least continuous so that we may always assume that it remains in a single coordinate chart. Under this assumption we can consider the pull back bundle $u^{\ast}\T\Nn$ to be trivial and the pulled back Levi-Civita connection is defined entirely by the form 
$$\omega^i_j :=\Gamma^i_{jk}(u)\ed u^k \in L^2(B_1,gl(m,\R)\otimes\wedge^1\T^{\ast}\R^2).$$
One has that $u$ is harmonic if 
$$\ed^{\ast}_{u^{\ast}\T\Nn}(\ed u)= -\Dl u^i - \Gamma^i_{jk}(u)(\pl{u^j}{x}\pl{u^k}{x}+\pl{u^j}{y}\pl{u^k}{y})= 0.$$
Here we have considered the connection as a covariant exterior derivative 
$$\ed_{u^{\ast}\T\Nn}:\Gamma(u^{\ast}\T\Nn\otimes \wedge^k \T^{\ast}\R^2) \to \Gamma(u^{\ast}\T\Nn \otimes \wedge^{k+1} \T^{\ast}\R^2)$$
and $\ed^{\ast}_{u^{\ast}\T\Nn}$ is the formal adjoint for $k=0$.\footnote{Given a vector bundle with a connection and a trivialisation the covariant exterior derivative is given simply by (where $e_i$ is our local frame, $\omega$ are our connection forms and
$v=v^i_{\al} \ed x^{\al} \otimes e_i$)
$${\ed}_{E} (v) = \ed v + \om\wedge v = (\ed v^i_{\al}\wedge\ed x^{\al} + \om^i_j\wedge (v^j_{\al}\ed x^{\al})) \otimes e_i.$$
If we have a metric on our vector bundle and we impose that the connection be compatible then when we choose an orthonormal (unitary) frame for our vector bundle, the connection forms are skew-symmetric (skew-hermitian) and the metric is trivial with respect to our trivialisation. From here the formal adjoint $\dv_E$ on sections $z\in \Gamma(E\otimes\wedge^1\T^{\ast}\R^2)$ ($z=(z^i_x\ed x^k + z^i_y\ed y)\otimes e_i$) is locally given by
$$\dv_E (z) = \dv z -\ast (\bar{\om}\wedge\ast z) = (\dv (z^i_x\ed x^k + z^i_y\ed y) - \ast (\bar{\om}^i_j \wedge \ast (z^i_x\ed x^k + z^i_y\ed y)))e_i$$ as can be checked directly. However we are not using an orthonormal frame for our trivialisation above so one must be more careful when computing the adjoint. It turns out that we do indeed have
$$\ed^{\ast}_{u^{\ast}\T\Nn}(\ed u)= -\Dl u^i - \Gamma^i_{jk}(u)(\pl{u^j}{x}\pl{u^k}{x}+\pl{u^j}{y}\pl{u^k}{y})$$as can be checked directly.}

In this case we also have 
$$\ed_{u^{\ast}\T\Nn}(\ed u)=\ed(\ed u^i) + \Gamma^i_{jk}(u)\ed u^k\wedge\ed u^j = 0.$$
It is thus unsurprising that by considering the domain as being complex we actually have 
$$\db_{u^{\ast}\T\Nn}(\de u) = \db \de u^i + \Gamma^i_{jk}(u)\db u^j \wedge \de u^k = 0.$$

Now we make another rash assumption: suppose that the connection is flat i.e. there exists a frame $S:B_1\to GL(m,\R)$ solving 
\begin{equation}\label{eqn flat}
S^{-1}\ed S + S^{-1}\om S = 0.
\end{equation}
In this case we can re-write our PDE with respect to $S$ and we actually have the coupled system 
$$\dv (S^{-1}\ed u) = 0 \,\,\,\,\,\,\,\,\,\,\,\,\,\text{and}\,\,\,\,\,\,\,\,\,\,\,\,\,\,\,\ed (S^{-1}\ed u) = 0,$$
or equivalently
$$\db (S^{-1}\de u) =0.$$
The assumption that the connection is flat is of course too strong (unless $\Nn$ is flat), therefore being able to find $S$ solving \eqref{eqn flat} is impossible. However we may utilise the complex domain here and consider the $\db$ problem by considering $\de u$ as a section of $u^{\ast}\T\Nn\otimes\C\otimes \wedge^{(1,0)}\T_{\C}^{\ast}\R^2$ and finding a holomorphic frame $S$ for the pulled back connection for the complexified bundle.

Unfortunately we have presupposed that $u$ is continuous in order to have these observations, however the potential lack of continuity of $u$ can be overcome by considering $(\Nn, h)$ to be isometrically embedded in some Euclidean space $\R^m$ at which point we can consider the critical points of the Dirichlet energy amongst maps in
$$W^{1,2}(B_1,\Nn):=\{v\in W^{1,2}(B_1,\R^m) : v(z)\in \Nn \,\,\,\text{for almost every}\,\,\,z\in B_1\}$$ and we can write the harmonic map equation as $(\Dl u)^{\top}=0$, or the projection of $\Dl u$ onto $\T_{u}\Nn$ is zero. Therefore we have
\begin{equation}\label{harmonic map}
\Dl u + \A(u)(u_x,u_x) + \A(u)(u_y,u_y)
\end{equation}
or in cordinates
$$\Dl u^i +\A^i_{jk}(u)(\pl{u^j}{x}\pl{u^k}{x}+\pl{u^j}{y}\pl{u^k}{y}) = 0$$
where $\A^i_{jk}$ are the components of the second fundamental form of $\Nn\emb\R^m$ (actually we have first extended the second fundamental form to a neighbourhood of $\Nn$, $\Nn_{\dl}$ so that $\A:\Nn_{\dl}\to \T^{\ast}\R^m \otimes\T^{\ast}\R^m \otimes \T\R^m$ with $\A^i_{jk}(u) = \A^i_{kj}(u)$ and the vector $\{\A^i_{jk}(u)\}_{i=1}^m$ is normal to $\Nn$ for each $j$ and $k$.) -- see \cite[Chapter 1]{helein_conservation} for details. The observation of Rivi\`ere was to let $\om^i_j : = (\A^i_{jk}(u)-\A^j_{ik}(u))\ed u^k\in L^2(B_1,so(m)\otimes\wedge^1\T^{\ast}\R^2)$ be our connection forms, and using the properties of $\A$ it can be checked that we have 
$$\dv_{\omega}(\ed u) = \dv (\ed u) -\ast(\om\wedge\ast\ed u) = 0$$
from which the higher regularity can be obtained by using the perturbed Coulomb gauge $A$ (the anti-symmetry of $\om$ { and the $L^{\infty}$ bound on $\A$} are essential here).
As is our wont, we can also write
$$\ed_{\omega}(\ed u)=\ed (\ed u) + \om\wedge\ed u =0$$
--this can be checked directly by using the properties of $\A$ and implies that the Hopf differential is holomorphic.\footnote{We have $\om^x u_y = \om^y u_x$ thus we can conclude (an exercise using the anti-symmetry of $\om$) that $\Dl u$ is perpendicular to both $u_x$ and $u_y$ (of course this is obvious for harmonic maps but works more generally when we assume $\om^x u_y = \om^y u_x$). The Hopf differential $$\psi:= (u^{\ast}h)^{(2,0)}:= (u_z,u_z)\ed z \wedge\ed z=\phi \ed z\wedge \ed z$$where $(,)$ is the Euclidean inner product extended complex linearly. Therefore $$\db \phi = 2(u_{z\zb},u_z) =0.$$ }

Therefore we also have 
$$\db_{\omega}(\de u) = \db (\de u) + \om^{\zb}\wedge \de u = 0$$
with $\om^{\zb} = (\A^i_{jk}(u)-\A^j_{ik}(u))\db u^k$.

{ Under the added assumption that the normal bundle is trivial (for instance when $\Nn$ is diffeomorphic to a sphere, or an orientable hypersurface in $\R^m$ etc), we can have a global normal frame $\{\nu_{K}\}_{K=N+1}^m$ for $\N\Nn$ that is $C^1$, and we can express $\A$ with respect to this frame via the Weingarten equation (note that this is independent of the choice of orthonormal normal frame):
$$\A^i_{jk}(z):=-\sum_K \pl{\nu^j_K}{z^k}\nu_K^i$$
where $z$ is the standard coordinate of $\R^m$ and $\nu$ is extended arbitrarily off $\Nn$. Thus we can write 
$$\om^i_j = (\ed\nu^i_K(u))\nu^j_K(u) - (\ed \nu^j_K(u))\nu^i_K(u) $$where we are summing over repeated indices (i.e. over $K$), thus $(\om^{\zb})^i_j= \db\nu^i_K(u)\nu^j_K(u) - \db \nu^j_K(u)\nu^i_K(u)$. Thus for a Hodge decomposition $\om = \ed a + \dv b$ with $b\in W_0^{1,2}$ we have
$$\Dl b^i_j = 2\ed\nu^i_K(u)\wedge \ed\nu^j_K(u)  $$
and therefore $\D b\in L^{2,1}$ by Theorem \ref{thm wente} with 
$$\|\D b\|_{L^{2,1}} \leq C\|\D u\|_{L^2}^2.$$ 
Notice that here we have $C=C(\sup |\A|)$.
As mentioned earlier if $\T\Nn$ is trivial we also have an easy proof using a $\db$ problem see \cite{helein_conservation}.

Following \cite{moser_regularity} we can utilise this idea for general $C^2$ target manifolds $\Nn$ by considering a smooth partition of unity $\{\chi_{\alpha}\}$ over $\Nn$ such that over the support of each $\chi_{\alpha}$ we know that the normal bundle of $\Nn$ is trivialised by $\{\nu_{\alpha,K}\}_{K=N+1}^m$. Setting $\nu_{\alpha,K}$ to be zero outside of the support of $\chi_l$ and defining 
$$\hat{\mu}_{\alpha,K}(z):=\chi_{\alpha}(z)\nu_{\alpha,K}(z)$$ we see that $\hat{\mu}$ is smooth over $\Nn$ with 
$$|\D \hat{\mu}_{\alpha,K}(z)|\leq \sup |\A|+\sup_{\alpha} |\D \chi_{\alpha}|\leq C(\sup |\A|)$$where the second inequality follows since $\sup |\A|$ uniformly controls the diameter $R$ of intrinsic balls over which the normal bundle can be trivialised. Thus our partition of unity can be constructed by smoothing out characteristic functions over balls of a fixed radius. Now define 
$$\mu_{\alpha,K}:=\hat{\mu}_{\alpha,K}(u)$$
and note that we can write 
$$\om^i_j:= (\ed \mu_{\alpha,K}^i)\nu_{\alpha,K}^j(u) - (\ed \mu_{\alpha,K}^j)\nu_{\alpha,K}^i(u)$$
where we are summing over both $\alpha$ and $K$. 

Again for a Hodge decomposition as above we have 
$$\Dl b^i_j =(\ed \mu_{\alpha,K}^i)\wedge (\ed\nu_{\alpha,K}^j(u)) - (\ed \mu_{\alpha,K}^j)\wedge(\ed\nu_{\alpha,K}^i(u))$$
and therefore $\D b\in L^{2,1}$ with 
$$\|\D b\|_{L^{2,1}} \leq C\|\D u\|_{L^2}^2$$
(again $C=C(\sup |\A|)$) and $\om$ satisfies condition $\dagger$ with 
$$\|\om\|_{L^2}+\|\D b\|_{L^{2,1}} \leq C\|\D u\|_{L^2}$$
whenever $\|\D u\|_{L^2} \leq 1$.  

Thus, Corollary \ref{cor_dbar} immediately gives Lipschitz estimates on $u$. The full regularity (along with smooth estimates) for harmonic maps follows from an easy boot-strapping argument using standard Calderon-Zygmund and Schauder estimates.  
\begin{theorem}[H\'elein]
Suppose $u:B_1 \to \Nn$ is a weakly harmonic map where $\Nn$ is a $C^l$ submanifold of $\R^m$ such that the second fundamental form is bounded with respect to the induced metric and $l\geq 2$. Then for all $\al\in (0,1)$ there exist $\eps = \eps(\sup |\A|)$ and $C=C(\Nn,\alpha)$ such that if
$$\|\D u\|_{L^2(B_1)} \leq \eps$$
then 
$$[\D^l u]_{BMO(B_{\fr12})} + \|u\|_{C^{l-1,\alpha}(B_{\fr12})} \leq C \|\D u\|_{L^2(B_1)}.$$
\end{theorem}
We also recover the following Energy convexity theorem in \cite{colding_minicozzi_width_ricci}, from which local uniqueness of harmonic maps follows easily in two dimensions. The proof can be found in  \cite[Appendix C]{colding_minicozzi_width_ricci} however now we can assume that $\Nn$ is $C^2$ with bounded second fundamental form and we do not need to make any assumptions on the tangent bundle. 

\begin{theorem}[Colding-Minicozzi]
Let $u,v\in W^{1,2}(B_1,\Nn)$ and suppose that $u$ is weakly harmonic map where $\Nn$ is a $C^2$ submanifold of $\R^m$ with bounded sencond fundamental form. Then there exists some $\eps = \eps (\sup|\A|)$ such that if $u-v\in W_0^{1,2}$ and 
$$\|\D u\|_{L^2(B_1)} \leq \eps$$
then 
$$\int_{B_1} |\D v|^2 - |\D u|^2 \geq \fr12 \int_{B_1} |\D(v-u)|^2.$$
\end{theorem}}

\subsection{Conformally invariant Lagrangians and an estimate of Lamm and Lin} 

Here we will recover (and marginally improve) the following result of Lamm and Lin (stated as a Corollary in \cite{lamm_lin}).
\begin{theorem}\label{thm: lamm_lin}
{ Let $\Nn\emb \R^m$ be an isometrically embedded, closed Riemannian manifold which is $C^2$ with bounded second fundamental form. Let $\gamma\in C^{1,1}(\Nn,\wedge^2\T^{\ast}\Nn)$ then every critical point in $W^{1,2}(B_1,\Nn)$ of the Lagrangian\footnote{It is known that all Lagrangians of this form are conformally invariant, moreover any Lagrangian that is conformally invariant and quadratic in the gradient takes this form when it is $C^2$ regular -- see \cite{gruter}.}$$F (u) = \int_{B_1} (\fr12|\D u|^2 \ed x\wedge\ed y + u^{\ast}\gamma)$$
solves 
$$\db_{\omega} (\de u) = 0$$
where 
$$\om^i_j:=(\ed \mu_{\alpha,K}^i)\nu_{\alpha,K}^j(u) - (\ed \mu_{\alpha,K}^j)\nu_{\alpha,K}^i(u)+ \ast\lambda^i_{jk}(u)\ed u^k$$ (see the previous section for definitions if necessary).} Here $\lambda^i_{jk}(u) = (\ed \pi_{\Nn}^{\ast}(\gamma))(u)(e_i,e_j,e_k)$ where $\pi_{\Nn}$ is the orthogonal projection onto $\Nn$ defined in some small tubular neighbourhood. Moreover $\om$ satisfies condition $\dagger_{\eps}$ with 
$$\eps\leq C\|\D u\|_{L^2(B_1)}.$$
Then whenever $E(u)$ is sufficiently small we have 
$$\||\D u|^2\|_{h^1(B_1)} \leq C\|\D u\|_{L^2(B_1)}^2$$ and locally smooth estimates for $u$ in terms of $E(u)$. 
\end{theorem}
\begin{proof}[Proof of theorem \ref{thm: lamm_lin}]
First of all notice that $\om\in L^2(B_1,u(m)\otimes \wedge^1\T^{\ast}\R^2)$ (since it is real and antisymmetric). We already know that (see \cite{riviere_inventiones} or \cite{helein_conservation}) 
$$\dv_{\omega}(\ed u) = 0$$ and it follows from the symmetries of $\A$ and $\lambda$ that 
$$\ed_{\omega}(\ed u) = \ed (\ed u) + \om\wedge\ed u = 0.$$
Therefore we can conclude that 
$$\db_{\omega}(\de u) = \db (\de u) + \om^{\zb}\wedge \de u = 0.$$

We now use Rivi\`ere's decomposition to find $A$ and $B$ solving
$$\ed A - A\om =\ast \ed B.$$
Inspecting the proof of \cite[Proposition 4.1]{lamm_lin} we have that 
$$\|\D B\|_{L^{2,1}(B_1)}\leq C\|\D u\|_{L^2(B_1)}.$$
Therefore 
$$\ed \om = \ed A^{-1}\wedge \ed A -  \ed (A^{-1}\ast\ed B)$$
from which we can conclude that 
$$\|\D \Dl^{-1} (\ed \om)\|_{L^{2,1}(B_1)} \leq C\|\D u\|_{L^2(B_1)}.$$ The rest of the proof follows from applying Corollary \ref{cor_dbar}.

\end{proof}

\section{Optimality of condition $\dagger$}\label{complex counter}

Here we present an example to show that the condition on the Hodge decomposition is sharp. 

Consider $\alpha : D \rightarrow \mathbb{C}^2$ given by $\alpha(z) = \frac{1}{z \log (\frac{e}{|z|})} (1, -i)\ed z \in L^2(D,\mathbb{C}^2\otimes \wedge^{(1,0)}\T^{\ast}_{\C}D)$ and we define $\om \in L^2(D,so(2)\otimes \wedge^1 \T^{\ast}\R^2)$ by 

\begin{equation*}
 \omega = \frac{1}{r^2\log(\frac{e}{r})}  \left( \begin{array}{cc}
0 & 1  \\
-1 & 0  \\
\end{array} \right) (y \ed x - x \ed y)
\end{equation*}
so that

\begin{equation*}
 \omega^{\zb} = \frac{i/2}{\bar{z}\log(\frac{e}{|z|})}  \left( \begin{array}{cc}
0 & -1  \\
1 & 0  \\
\end{array} \right) \ed \zb.
\end{equation*}
A short { calculation} yields that 

\begin{equation*}
\frac{\partial \alpha}{\partial \bar{z}}\ed \zb \wedge \ed z = \frac{1/2}{(|z| \log (\frac{e}{|z|}))^2} (1, -i)\ed \zb \wedge \ed z = -\omega^{\zb} \wedge \alpha.
\end{equation*}
Therefore $\db_{\om}\alpha = 0$ but $$\al \notin L^{\infty}_{loc}(D)$$and$$\frac{\partial \alpha}{\partial z} = \left( \frac{-1}{z^2  \log (\frac{e}{|z|})} + \frac{1/2}{(z \log (\frac{e}{|z|}))^2 }\right) (1, -i) \notin L_{loc}^1(D)$$
thus Theorem \ref{dbar} cannot hold in this case. 

It is easy to see that $$\om = \ast\ed u\twomat{0}{1}{-1}{0}$$ where $u=\log \log (\fr{e}{r})$ is the Frehse example, thus setting $b=\ast u$ one can easily check that $\dv b=\ast\ed u \in L^{2,q}$ for all $q>1$ but $\dv b\notin L^{2,1}$, hence condition $\dagger$ is sharp in this sense.  

\section{Proof of the regularity result, Corollary \ref{cor_dbar}
}\label{proof_cor_dbar}
Suppose that $\om$ satisfies $\dagger_{\eps}$ ($\eps$ given by Theorem \ref{dbar}). 
We check that 
$$\db(S^{-1}\al) = -S^{-1}\db S S^{-1}\wedge \al - S^{-1}\om^{\zb}\wedge \al = 0.$$
Therefore $\al = S h$ for some holomorphic $h$ and the estimates follow by standard theory. A simple covering argument completes the first part of the proof.

The proof of the final assertion (that $|\al|^2 \in h^1$) follows from the following fact that is easily verified: Given a holomorphic function $h \in L^2(D)$, then $|h|^2 \in h^1(D)$ with 
$$\||h|^2\|_{h^1(D)} \leq C\|h\|_{L^2(D)}^2.$$
To see this first notice that $h= f_z$ for some holomorphic $f=f_1+if_2\in W^{1,2}(D)$ (this follows from the Poincar\'e lemma, for instance). Thus we have (since $f$ is holomorphic) 
$$|h|^2 = |f_z|^2 = -\ast (\ed  f_1 \wedge \ed f_2) \in h^1(D)$$
by the main result in \cite{clms} (coupled with an extension argument).

In our case we have $\al = S h$ so that there exists some $C$ with 
$$C^{-1}|h|^2\leq |\al|^2 \leq C|h|^2.$$
Thus we have 
$$\||\al|^2\|_{h^1(B_1)} \leq C\||h|^2\|_{h^1(B_1)} \leq C\|h\|_{L^2(D)}^2 \leq C\|\al\|_{L^2(D)}^2. $$ 

\section{Proof of the existence of a holomorphic gauge, Theorem \ref{dbar}}
We start by finding the Coulomb frame associated to $\ed a$; using Theorem \ref{theorem Coulomb} we can find $P\in W^{1,2}(D,U(m))$ and $\eta \in W_0^{1,2}(D, u(m)\otimes \wedge^2 \T^{\ast} D)$ such that 
\begin{equation}\label{frameda}
P^{-1}\ed P + P^{-1}\ed a P = \dv \eta
\end{equation}
and 
$$\|\D P\|_{L^2(D)} + \|\eta\|_{W^{1,2}(D)} \leq C\|\ed a\|_{L^2(D)} \leq C\|\om\|_{L^2(D)}.$$

Thus on $D$ we have a solution to 
$$\Dl \eta = \ed \bar{P}^{T} \wedge \ed P + \ed (\bar{P}^{T}\ed a P).$$
Or, in coordinates we have
$$\Dl \eta^i_j = \ed \bar{P}^k_i \wedge \ed P^k_j + \ed (\bar{P}^l_i P^k_j)\wedge \ed a^l_k.$$Again we sum over repeated indices here so that we sum over $k$ in the first term, and both $k$ and $l$ in the second. 

The estimates from Theorem \ref{thm wente} give 
$$\|\D \eta\|_{L^{2,1}(D)} \leq C\|\om\|_{L^2(D)}^2.$$

Now we check how $P$ transforms $\om$, by \eqref{frameda} 
we have
\begin{equation}\label{frameom}
P^{-1}\ed P + P^{-1}\om P = \omega_P= \dv \eta + P^{-1}\dv b P \in L^{(2,1)}(D) 
\end{equation}
and 
$$\|\dv \eta + P^{-1}\dv b P\|_{L^{2,1}(D)} \leq C\eps.$$
We can see here the significance of condition $\dagger$, essentially it allows us to change the connection forms so that the whole of the transformed  connection lies in $L^{2,1}$. 

We can now take the $(0,1)$-part of \eqref{frameom} 
to give 
\begin{equation}
P^{-1}\db P + P^{-1}\om^{\zb} P = \ast\db\ast \eta + \ast P^{-1}\db \ast b P \in L^{(2,1)}(D) 
\end{equation}
which after applying Theorem \ref{theorem KM L21} (by setting $\eps$ small enough)
gives us the existence of some $Q\in C^0\cap W^{1,(2,1)}_{loc}(D,GL(k,\C))$ satisfying
$$\db Q = -(\ast\db\ast \eta + \ast P^{-1}\db \ast b P)Q,$$
$$\|\rm{dist}(Q,Id)\|_{L^{\infty}(D)} \leq \fr13$$
and for any $U\cemb D$ there exists $C=C(U)<\infty$ such that
$$\|\D Q\|_{L^{2,1}(U)} \leq C\|\om\|_{L^2(D)}.$$

Thus we have 
$$P^{-1}\db P + P^{-1}\om^{\zb} P = -\db Q Q^{-1}$$ 
and therefore setting $S=PQ \in L^{\infty}(D,GL(k,\C))\cap W_{loc}^{1,2}(D, GL(k,\C))$ we have
$$\db S = -\om^{\zb} S$$
with the desired estimates. 

\section{A few remarks}
We could generalise this, and simply consider maps $v\in L^2(B_1, \C^m \otimes \wedge^1 \T^{\ast}\R^2)$ solving 
$$\dv_{\omega}(v)=0$$
and 
$$\ed_{\omega}(v) =0$$
for some connection $\om\in L^2(B_1,u(m)\otimes\wedge^1 \T^{\ast}\R^2)$. As above we can check that $v^{(1,0)}\in L^2(B_1,\C^m\otimes \wedge^{(1,0)}\T^{\ast}_{\C}\C)$ solves 
$$\db_{\om}(v^{(1,0)}) = 0.$$
Now we can ask, under what conditions can we find a holomorphic change of frame $S$ as in Theorem \ref{dbar} 
in order to conclude $v\in (L^{\infty}\cap W^{1,2})_{loc}$. In general we cannot do this unless $\om$ satisfies condition $\dagger$ because of the counter-example presented in section \ref{complex counter}. 
However we are still free to change our frame via a map $P\in W^{1,2}(B_1,U(m))$, and writing $\ti{v}^{(1,0)}:=P^{-1}v^{(1,0)}$ we have 
$$\db_{\omega_P} (\ti{v}^{(1,0)}) = 0$$
where 
$$P^{-1}\ed P + P^{-1}\om P = \omega_P.$$
Now we can ask whether $\omega_P$ satisfies condition $\dagger$? In particular this is the case if $\ed (\omega_P) = 0$ (the `opposite' of what is achieved in considering a Coulomb frame) or $\ed (\omega_P) \in \h^1$. More generally this is true if 
$$\D \Dl^{-1} (\ed (\omega_P)) = \D \Dl^{-1}(\ed P^{-1}\wedge\ed P + \ed (P^{-1}\om P)) \in L^{2,1}.$$ 
Therefore because of Theorem \ref{thm wente} we can reduce this condition to being able to find a frame $P$ such that 
$$\D \Dl^{-1} (\ed (P^{-1}\omega P)) \in L^{2,1}.$$

The bottom line here is the following:
\begin{theorem}
Let $\om \in L^2(D,u(m)\otimes\wedge^1 \T^{\ast}\R^2)$, and suppose there exists a change of frame $P\in W^{1,2}(B_1,U(m))$ such that 
$$\omega_P = P^{-1}\ed P + P^{-1}\om P$$
satisfies condition $\dagger$. Then there exists $\eps >0$ such that whenever $\omega_P$ satisfies condition $\dagger_{\eps}$  
there exists a change of frame $S\in L^{\infty}\cap W^{1,2}(D_{\fr23},Gl(k,\C))$ such that $$\db S = -\om^{\zb} S$$
with
$$\|\rm{dist}(S,U(m))\|_{L^{\infty}(D)} \leq \fr13$$
and
$$ \|\D S\|_{L^2(D_{\fr12})} \leq C\|\om\|_{L^2}.$$
\end{theorem}

\subsection{Conformal immersions of surfaces into Riemannian manifolds} 

In this section we consider a conformal immersion $u:B_1 \to \Nn \emb \R^m$ with bounded area and mean curvature $H:B_1\to \R^m$. It is well known that $u$ solves 
\begin{equation}\label{eqn conf_imm}
\tau (u)  = H (|u_x|^2 + |u_y|^2)
\end{equation}
where $\tau(u)^i:= -\Dl u^i - \A^i_{jk}\D u^k \cdot \D u^j$ is the tension field of $u$. Since $u$ is conformal we also have
$$|u_x|^2 - |u_y|^2 = \langle u_x, u_y \rangle  = 0$$
and we consider the surface $\Sigma = (B_1, \rho^2(\ed x^2 + \ed y^2))$ so that $u$ is an isometry $u:\Sigma \to u(B_1)$ and $\rho = |u_x|$.
When $H\in L^2(\Sigma)$ we can find $\om\in L^2(B_1,so(m)\otimes\T^{\ast}\R^2)$ such that \eqref{eqn conf_imm} can be written 
\begin{equation*}
\dv_{\omega}(\ed u) = \ed_{\omega}(\ed u) = \db_{\omega}(\de u) = 0.
\end{equation*}
To see this let 
\begin{eqnarray*}
 \om^i_j &:=& (\omega_{\Nn})^i_j + (\omega_H)^i_j\\
&:=& {(\ed \mu_{\alpha,K}^i)\nu_{\alpha,K}^j(u) - (\ed \mu_{\alpha,K}^j)\nu_{\alpha,K}^i(u)}+ (H^i\ed u^j - H^j\ed u^i)
\end{eqnarray*}
so that 
\begin{eqnarray*}
\ed_{\omega}(\ed u) &=& \ed (\ed u^i) + \om^i_j\wedge \ed u^j \\
&=&  (\A^i_{jk}(u) - \A^j_{ik}(u))\ed u^k\wedge\ed u^j +  (H^i\ed u^j - H^j\ed u^i)\wedge \ed u^j \\
&=& 0,
\end{eqnarray*}
and 
\begin{eqnarray*}
\dv_{\omega}(\ed u) &=& \dv (\ed u^i) -\ast (\om^i_j \wedge\ast\ed u^j) \\
&=& -\Dl u^i - (\A^i_{jk}(u) - \A^j_{ik}(u))\ast(\ed u^k\wedge\ast\ed u^j) - \ast ((H^i\ed u^j - H^j\ed u^i)\wedge \ast\ed u^j) \\
&=& \tau(u) - H(|u_x|^2 + |u_y|^2) = 0,
\end{eqnarray*}
which together imply that 
\begin{eqnarray*}
\db_{\om}(\de u) &=& \db \de u + \om^{\zb}\wedge \de u \\
&=& \fr14 \Dl u^i \ed \zb\wedge\ed z + (\A^i_{jk}(u) - \A^j_{ik}(u))\db u^k \wedge \de u^j + (H^i\db u^j - H^j\db u^i)\wedge\de u^j \\
&=& -\fr14 \dv_{\omega}(\ed u)\ed \zb \wedge\ed z + \fr{1}{2}\ed_{\omega}(\ed u) = 0.
\end{eqnarray*}

We therefore see that for $H\in W^{1,2}(B_1 , \R^m)$ with $H\in L^2(\Sigma)$ there exists $\eta = \eta (\Nn, m, \eps)$ such that if 
$$Area(u(B_1)) + \|\D H\|_{L^2(B_1)} + \|H\|_{L^2(\Sigma)} \leq \eta$$
then $\om$ satisfies $\dagger_{\eps}$. 

\begin{problem}
The requirement that $H\in W^{1,2}$ does not seem to be natural, therefore one could ask whether only considering $H\in L^2(\Sigma)$ and 
$$Area(u(B_1)) + \|H\|_{L^2(\Sigma)}\leq \eta$$ is enough to find a holomorphic gauge? Using the previous section this would amount to finding a change of frame $P\in W^{1,2}(B_1,SO(m))$ such that 
$$P^{-1}\omega_H P$$ satisfies condition $\dagger$ with 
$$\|\D \Dl^{-1} (\ed (P^{-1}\om_H P))\|_{L^{2,1}(B_1)} \leq \|H\|_{L^2(\Sigma)}.$$ In the next section we essentially show that this is possible with $P=Id$ under the (strong) assumptions that the mean curvature is parallel, and $\Nn = \R^m$. 

\end{problem}

\subsubsection{Parallel mean curvature}
Here we consider the situation where $\Nn = \R^m$ and $u(B_1)$ has parallel mean curvature (PMC). This condition means that $\Db H = 0$ where $\Db$ is the induced connection on the normal bundle of $\Sigma$. This is equivalent to the condition that $\pl{H}{x}$ and $\pl{H}{y}$ are tangent to $\Sigma$ and therefore we may conclude that $|H|^2$ is a constant (since $H$ is normal to $\Sigma$). Moreover we will use the expressions 
$$\pl{H}{x} = H_x = \langle H_x , u_x \rangle \fr{u_x}{\rho^2} + \langle H_x , u_y \rangle \fr{u_y}{\rho^2},$$
$$\pl{H}{y} = H_y = \langle H_y , u_x \rangle \fr{u_x}{\rho^2} + \langle H_y , u_y \rangle \fr{u_y}{\rho^2},$$
which hold since we have PMC. Also 
$$\langle H_x , u_x \rangle = - \langle H, u_{xx} \rangle,$$
$$\langle H_y , u_y \rangle = - \langle H, u_{yy} \rangle,$$and 
$$\langle H_x , u_y \rangle = \langle H_y, u_x\rangle$$
which hold simply because $H$ is normal.

Now, we still have 
\begin{equation*}
\dv_{\omega_H}(\ed u) = \ed_{\omega_H}(\ed u) = \db_{\omega_H}(\de u) = 0
\end{equation*}
and a computation using the fact that we have PMC gives 
\begin{eqnarray*}
(\ed \om_H)^i_j &=& \ed H^i \wedge \ed u^j - \ed H^j \wedge \ed u^i \\
&=& (H^i_xu^j_y  - H^j_x u^i_y + H^j_y u^i_x - H^i_yu^j_x)\ed x\wedge\ed y\\
&=& (\langle H_x, u_x \rangle \fr{u_x^iu_y^j - u_x^ju_y^i}{\rho^2} + \langle H_y ,u_y \rangle \fr{u_x^iu_y^j - u_x^ju_y^i}{\rho^2} )\ed x \wedge\ed y \\
&=& -\fr{\langle H, \Dl u\rangle}{\rho^2} \ed u^i \wedge\ed u^j\\ 
&=&2|H|^2\ed u^i \wedge\ed u^j.
\end{eqnarray*}
Moreover we also have 
$$\dv \omega_H = 0$$
weakly since 
\begin{eqnarray*}
(\ed \ast \omega_H)^i_j &=& H^i\ed\ast\ed u^j - H^j\ed\ast\ed u^i + \ed H^i \wedge\ast\ed u^j - \ed H^j \wedge \ast \ed u^i \\
&=& (H^i_x u^j_x  - H^j_xu_x^i + H^i_y u^j_y- H^j_yu_y^i) \ed x\wedge \ed y \\
&=& (\langle H_x,u_y\rangle \fr{u_y^iu^j_x - u_y^ju_x^i}{\rho^2} +\langle H_y, u_x\rangle \fr{u_x^iu_y^j - u^j_xu_y^i}{\rho^2})\ed x\wedge \ed y \\
&=& 0.
\end{eqnarray*}
This tells us two things: Firstly that $\om$ satisfies condition $\dagger_{\eps}$ when 
$$|H|^2 Area(u(B_1))$$ is sufficiently small,
and secondly we also have that $\Dl u$ is a sum of Wente terms by writing
$$\omega_H=\ast \ed \eta$$and 
$$-\Dl u = \ast (\ed \eta \wedge \ed u).$$
The latter fact is obvious when $m=3$.

\section{Wente estimates and changes of frame}
We will use the following well known estimate, which follows from the results of \cite{clms} and \cite{feff_stein} but in a simpler form is due to \cite{wente}. We also use implicitly here the continuous embedding $W^{1,1}(B_1)\emb L^{2,1}(B_1)$ when $B_1 \In \R^2$. A proof of this fact can be found in \cite{helein_conservation}. \begin{theorem}\label{thm wente}
Suppose that $\phi\in W_0^{1,2}$ is a solution to 
$$\Dl \phi = \ast(\ed a \wedge \ed b)$$
with $a,b \in W^{1,2}(B_1,\R)$ on the unit disc $B_1\In\R^2$. Then $\phi \in W^{2,1}$ with 
$$\|\D^2 \phi\|_{L^1(B_1)}+\|\D \phi\|_{L^{2,1}(B_1)} + \|\phi\|_{C^0(B_1)} \leq C\|\D a\|_{L^2(B_1)}\|\D b\|_{L^2(B_1)}.$$
\end{theorem}

We state here the following classical theorem of Koszul-Malgrange \cite[Theorem 1]{KM}, below we use $\frak{G}$ to denote a complex Lie group and $\frak{g}$ its Lie algebra. 
\begin{theorem}\label{theorem KM}
Let $U\In\C^n$ be open and $\al\in C^{\infty}(U,\frak{g}\otimes \wedge^{(0,1)}\T_{\C}^{\ast}\C^n)$. Then, for any open $V\In U$ there exists $f\in C^{\infty}(V,\frak{G})$ solving 
$$f^{-1}\db f = \al$$
if and only if 
\begin{equation}\label{eqn zerotwo curvature}
\db \al  + [\al,\al] = 0\end{equation}
on $V$. 
\end{theorem}

In the case that $\al$ is the $(0,1)$-part of a connection form, the expression \eqref{eqn zerotwo curvature} is precisely the $(0,2)$ part of the curvature. As we have mentioned previously, in the case that $n=1$ it is clear that such an $f$ always exists since the condition \eqref{eqn zerotwo curvature} is vacuously true. The following is a non-smooth version of Theorem \ref{theorem KM} for $n=1$ and $\frak{G}=GL(m,\C)$, the proof of which can be found in \cite{helein_conservation}. 

\begin{theorem}\label{theorem KM L21}
Let $\om^{\zb}\in L^{2,1}(D, gl(k,\C)\otimes \wedge^{(0,1)}\T^{\ast}_{\C}\C)$ then there exists $\eps >0 $ such that whenever $\|\om\|_{L^{2,1}(D)}\leq \eps$ there is a $Q\in C^0\cap W^{1,(2,1)}_{loc}(D,GL(k,\C))$ such that 
$$\db Q = -\omega^{\zb}Q$$
and 
$$\|\rm{dist}(Q,Id)\|_{L^{\infty}(D)} \leq \fr13.$$
\end{theorem}
It follows from standard estimates for harmonic maps that the smallness condition on $\|\om\|_{L^{2,1}}$ cannot be dropped in order that we keep the $L^{\infty}$ estimate on $Q$. For instance if one considers a sequence of harmonic maps $\{u_n\}:B_1\to S^2$ with uniformly bounded energy, that undergoes bubbling, then one has $\|\D u_n\|_{L^{2,1}}$ is uniformly bounded.\footnote{In this instance we know there exist functions $B^i_j$ such that $$\Dl u = \ast (\ed B^i_j \wedge \ed u^j)$$ and $$\|\D B^i_j\|_{L^2}\leq C\|\D u\|_{L^2}.$$See \cite{helein_conservation}.} We also know that $\al_n = \de u_n$ solves 
$$\db_{\omega_n}\al_n = 0$$
with $\|\om_n\|_{L^{2,1}}\leq C\|\D u_n\|_{L^{2,1}}$ so that if one could find such maps $Q_n$, bounded in $L^{\infty}$ then we could conclude that $\|\D u_n\|_{L^{\infty}}$ is uniformly bounded, contradicting the assumption that the maps undergo bubbling. 

We would also like to recall some results about existence of Coulomb (or Uhlenbeck see \cite{uhlenbeck_lp}) gauges. We provide a proof of the following, communicated to us by Ernst Kuwert, as we have not seen it elsewhere, although similar Theorems are proved in \cite{riviere_inventiones} and \cite{schikorra_frames}.  
\begin{theorem}\label{theorem Coulomb}
Let $\om\in L^2(B_1, u(m)\otimes\wedge^1 \T^{\ast}\R^2)$ then we can find maps $P\in W^{1,2}(B_1,U(m))$ and $\eta \in W_0^{1,2}(B_1, u(m) \otimes \wedge^2 \T^{\ast} \R^2)$ such that 
$$P^{-1}\ed P + P^{-1}\om P = \dv \eta. $$
Moreover  
$$ \|\D \eta\|_{L^2(B_1)} \leq \|\om\|_{L^2(B_1)}$$
and 
$$\|\D P\|_{L^2(B_1)} \leq 2\|\om\|_{L^2(B_1)}.$$
\end{theorem}

\begin{proof}
The Coulomb gauge $P$ is found by minimising the following energy
$$E(P): = \int_{B_1} |P^{-1}\ed P + P^{-1}\om P|^2 = \int_{B_1} |\ed P + \om P|^2,$$
which effectively is trying to minimise the $L^2$ distance of our connection to the exterior derivative. 

Clearly $P\equiv Id$ is admissible so that we choose a minimising sequence $\{P_n\}$ with
$$E(P_n) \leq E(Id) = \|\om\|_{L^2}^2.$$
We also have that 
\begin{eqnarray*}
\|\D P_n\|_{L^2(B_1)}^2 &=& E(P_n) - \|\om\|_{L^2}^2 - \int_{B_1} 2\langle \ed P_n, \om P_n\rangle \\
&\leq& \|\om\|_{L^2}^2 -\|\om\|_{L^2}^2 + 2\eps \|\D P_n\|_{L^2(B_1)}^2 + \fr{1}{2\eps}\|\om\|_{L^2}^2
\end{eqnarray*}
by Young's inequality. Therefore 
$$\|\D P_n\|_{L^2(B_1)}^2 \leq 4\|\om\|_{L^2}^2$$
and we can find $P\in W^{1,2}(B_1,gl(m,\C))$ such that $P_n\wtu P$ in $W^{1,2}$. This tells us that in particular, $\ed P_n \wtu \ed P$ in $L^2$ and $P_n \to P$ pointwise almost everywhere. It follows that $P\in W^{1,2}(B_1, U(m))$ and that  
\begin{eqnarray*}
E(P_n) - E(P) &=& \int_{B_1} |\D P_n|^2 - |\D P|^2 + |\om|^2 - |\om|^2 +  2\langle \ed P_n, \om P_n \rangle - 2\langle \ed P ,\om P\rangle \\
&=& \int_{B_1} |\D P_n|^2 - |\D P|^2 +2\langle \ed P_n, \om(P_n-P)\rangle + 2\langle \ed P_n -\ed P ,\om P\rangle. 
\end{eqnarray*}
Now, 
$$\int_{B_1} |\omega(P_n-P)|^2= \int_{B_1} \langle \om P_n ,\om P_n \rangle +\langle \om P,\om P \rangle -2\langle \om P_n,\om P \rangle\to 0$$
as $n\to \infty$ by Lebesgue's dominated convergence theorem.

We also have 
$$\int_{B_1} \langle \ed P_n -\ed P ,\om P\rangle \to 0$$
and 
$$\liminf \|\D P_n\|_{L^2}^2 \geq \|\D P\|_{L^2}^2.$$
Putting these things together yields the lower semi-continuity of $E$ and 
$$E(P)\leq \liminf E(P_n)\leq \|\om\|_{L^2}^2.$$

Let $\omega_P := P^{-1}\ed P + P^{-1}\om P$ and using the fact that $P$ is a critical point of $E$ we consider variations $\phi \in C^1(\bar{B}_1, u(m))$ and we get 
$$\fr{\ed}{\id t}\int_{B_1} |\ed(Pe^{t\phi}) + \om Pe^{t\phi}|^2 \,\,\vlinesub{t}=0.$$
Using 
$$|\ed(Pe^{t\phi}) + \om Pe^{t\phi}|^2 = |P\ed(e^{t\phi}) + P\omega_P e^{t\phi}|^2 = |\ed(e^{t\phi})+\omega_P e^{t\phi}|^2$$
it can be easily checked that 
$$0=\fr{\ed}{\id t}\int_{B_1} |\ed(Pe^{t\phi}) + \om Pe^{t\phi}|^2 \,\,\vlinesub{t} = 2\int_{B_1}\langle \ed \phi , \om_P\rangle $$
for any $\phi\in W^{1,2}(B_1)$ by approximation. In fact we can allow $\phi \in L^2_{loc}(B_1)$ with $\D \phi \in L^2(B_1)$ by approximating with $\phi_r(x) = \phi (rx)$ as $r\uu 1$.

By a linear Hodge decomposition we know there exists $\al \in W_0^{1,2}(B_1,u(m))$, $\eta \in W_0^{1,2}(B_1, u(m)\otimes \wedge^2 \T^{\ast}\R^2)$ and a harmonic one form $h\in L^2(B_1,u(m)\otimes \wedge^1\T^{\ast}\R^2)$ - see for instance \cite{morrey} - such that 
$$\omega_P = \ed \al + \dv \eta + h.$$
Let $\gamma \in L^2_{loc}(B_1)$ be the harmonic function given by 
$$\ed \gamma = h$$
giving $\D \gamma \in L^2(B_1).$ Thus we have that $\al + \gamma$ is an admissible test function. 
This fact, coupled with 
$$\int_{B_1} \langle \ed \phi ,\dv \eta\rangle = 0 $$
for all such $\phi$ gives 
$$\int_{B_1} \langle \ed \phi, \ed \al + h\rangle = 0,$$
and therefore 
$$0=\int_{B_1} \langle \ed \al + h , \ed \al + h \rangle = \int_{B_1} |\ed \al + h|^2 = \int_{B_1} |\ed \al|^2 + |h|^2.$$
Therefore $$\ed P + \om P = P\dv \eta$$ giving the final estimate and completing the proof. 
\end{proof}

\bibliographystyle{plain}

\sc Department of Mathematics, South Kensington Campus, Imperial College London, LONDON, SW7 2AZ, UK, ben.g.sharp@gmail.com
\end{document}